\documentclass[12pt,twoside]{amsart}
\usepackage[margin=3cm]{geometry}
\usepackage[colorlinks=false]{hyperref}
\usepackage[english]{babel}
\usepackage{graphicx}%,titling}%有了这个titling之后反而不能自动生成标题
\usepackage{float}
\usepackage{amsmath,amsfonts,amssymb,amsthm}
\usepackage{lipsum}
\usepackage[T1]{fontenc}
\usepackage{fourier}
\usepackage{color}
\usepackage{esint}
\usepackage{caption}
\usepackage{ccicons}
\makeatletter
\def\blfootnote{\xdef\@thefnmark{}\@footnotetext}
\makeatother

\newcommand\ccnote{
    \blfootnote{Jie Zhou is supported by NSFC No.12301077. Jintian Zhu is partially supported by National Key R\&D Program of China with grant no. 2023YFA1009900 as well as the start-up fund from Westlake University.}
}
\allowdisplaybreaks%用于自动断页
\usepackage[export]{adjustbox}
\numberwithin{equation}{section}
\usepackage{setspace}
\renewcommand{\le}{\leqslant}
\renewcommand{\leq}{\leqslant}
\renewcommand{\ge}{\geqslant}
\renewcommand{\geq}{\geqslant}
\renewcommand{\mathbb}{\varmathbb}
\usepackage{fancyhdr}
\newtheorem{theorem}{Theorem}[section]

\newtheorem{corollary}[theorem]{Corollary}

\newtheorem{remark}[theorem]{Remark}

\newtheorem{ques}[theorem]{Question}
\newtheorem{conj}[theorem]{Conjecture}
\usepackage{comment}
\usepackage{enumerate}
\usepackage{hyperref}
\usepackage[alphabetic, msc-links, backrefs]{amsrefs}

    \newcommand{\dist}{\operatorname{dist}}
 
 \newcommand{\area}{\operatorname{area}}

\usepackage{amsthm}

\allowdisplaybreaks%用于自动断页

\input epsf
\def\begfig {
\begin{figure}
\small }
\def\endfig {
\normalsize
\end{figure}
}

\DeclareMathOperator{\im}{Im}
\DeclareMathOperator{\vol}{vol}
\DeclareMathOperator{\Ric}{Ric}
\DeclareMathOperator{\bRic}{biRic}
\DeclareMathOperator{\Sc}{Sc}
\DeclareMathOperator{\Rm}{Rm}

\DeclareMathOperator{\diam}{diam}
\begin{document}
%\firstpage{}\volume{}\volumeyear{}\copyrightyear{}\doiyear{}\doi{}\received{XXX}\revised{XXX}
\thispagestyle{empty}

\ccnote

\title[]{Optimal volume bound and volume growth  for Ricci-nonnegative manifolds with positive Bi-Ricci curvature}

\author[J. Zhou]{Jie Zhou}
\address[Jie Zhou]{School of Mathematical Sciences, Capital Normal University, 105 West Third Ring Road North, Haidian District, Beijing 100048, People's Republic of China}
\email{zhoujiemath@cnu.edu.cn}

\author[J. Zhu]{Jintian Zhu}
\address[Jintian Zhu]{Institute for Theoretical Sciences, Westlake University, 600 Dunyu Road, Hangzhou, Zhejiang 310030, People's Republic of China}
\email{zhujintian@westlake.edu.cn}
\maketitle

\noindent \textbf{Abstract.}  In this paper, we prove the optimal volume growth for complete Riemannian manifolds $(M^n,g)$ with nonnegative Ricci curvature everywhere and bi-Ricci curvature bounded from below by $n-2$ outside a compact set when the dimension is less than eight. This answers a question \cite[Question 1]{Antonelli-Xu2024} proposed by Antonelli-Xu in dimensions six and seven. As a by-product, we also prove an analogy of Gromov's volume bound conjecture \cite[Open Question 2.A.(b)]{Gro86} under the condition of positive bi-Ricci curvature.
\vskip0.3cm

  \vspace{0.5cm}

%\noindent \textbf{MSC2020.} %49Q20, 35B65

%%      -------------------------------------------------------------------------------
%%      -------------------------- BEGIN ARTICLE ----------------------------
%%      -------------------------------------------------------------------------------
%% Authors, copy the body of your paper here

\tableofcontents
% The domination of volume by lower bound of curvature on a Riemannian manifold is a classical topic in Riemannian Geometry. For example, 

\section{Introduction}
The domination of volume  by lower bound of curvature on Riemannian manifolds is a classical topic in Riemannian geometry.  The Bishop-Gromov volume comparison theorem implies that the volume of a Riemannian manifold $(M^n,g)$ with  positive Ricci lower bound $\Ric\ge (n-1)g$ is no greater than the volume of the unit $n$-sphere. Applied to non-compact manifolds it also yields that the  volume ratio 
$$\Theta^n(p,r)=\frac{\vol(B_r(p))}{\omega_n r^n},\,p\in M,$$
of a Riemannian manifold with nonnegative Ricci curvature is monotone non-increasing as $r$ increases.  For weaker curvature assumption, 
it is well-known that
positive scalar curvature alone is in general not enough to dominate the volume.  The counterexample can be constructed easily based on the surgery theory developed by Gromov-Lawson \cite{Gromov-Lawson1980} and Schoen-Yau \cite{Schoen-Yau1985} independently.
On the other hand, there are also many evidences
revealing that positive scalar curvature can provide desired volume control under extra assumptions on lower bound of the Ricci curvature. A typical result is Bray's football theorem \cite{Bray97} saying that there exists a sharp constant $\epsilon_0\approx 0.134$ such that any $3$-manifold $(M^{3},g)$ with scalar curvature $\Sc\ge 6$ and Ricci curvature $\Ric\ge 2 \epsilon_0g$ satisfies the optimal volume estimate $\vol(M)\le |\mathbb S^3|$.  In the same spirit,  Gromov proposed the following volume bound conjecture \cite[Open Question 2.A.(b)]{Gro86}.

 \begin{conj}[Volume Bound Conjecture]\label{volume bound conjecture}  There exists a dimensional constant $C(n)$ such that for any Riemannian manifold $(M^n,g)$ with $\Ric\ge 0$ and   $\Sc\ge 2$, there holds 
 \begin{align*}
 \sup_{p\in M, R>0}\frac{\vol(B_R(p))}{R^{n-2}}\le C(n).
 \end{align*}
 \end{conj} 
\begin{remark}
    More precisely, Gromov sketched \cite[2.b]{Gro86} the idea of proving the volume bound in the case of nonnegative sectional curvature and ask whether it can be generalized to the case of nonnegative Ricci curvature.
\end{remark}

For manifolds with nonnegative Ricci curvature there is also a closely related question \cite[Problem 9]{Yau92} proposed by Yau asking whether a complete non-compact manifold $(M^n,g)$ with $\Ric\ge 0$ satisfying 
$$\lim_{R\to \infty}\frac{\int_{B_R(p)}\Sc(x)\text{ d}\vol}{R^{n-2}}<\infty \, ?$$
In the same fashion of Yau's question, Gromov's conjecture  can be asked in an asymptotic version concerning the volume growth control of a complete non-compact manifolds with non-negative Ricci curvature and positive scalar curvature. The following volume growth conjecture is made explicit in \cite{Zhubo22, Wei-Xu-Zhang2024}.

\begin{conj}[Volume Growth Conjecture]\label{volume growth conjecture}
There exists a constant $C(n)$ such that for any complete non-compact Riemannian manifold $(M^n,g)$ with $\Ric\ge 0$ and $\Sc\ge 2$, there holds 
 \begin{align*}
\Theta^{n-2,*}(\infty):=\limsup_{R\to \infty}\frac{\vol(B_R(p))}{R^{n-2}}\le C(n).
 \end{align*}
 \end{conj} 
  By applying Cheeger-Colding theory, Zhu \cite{Zhubo22} proved the finiteness of $\Theta^{n-2,*}(\infty)$ for any complete non-compact Riemannian manifold $(M^n,g)$ with $\Ric\ge 0$ and $\Sc\ge 2$ as well as some extra non-collapsing conditions.  By applying level-set method for certain harmonic functions,  Munteanu-Wang \cite[Corollary 1.5]{MW22} successfully verified Gromov's volume bound conjecture as well as the volume growth conjecture above for $n=3$.  Later, Chodosh-Li-Stryker \cite{CLS23}  gave a new approach for Munteanu-Wang's linear  volume  growth result by applying the $\mu$-bubble method introduced by Gromov and the almost splitting theorem in Cheeger-Colding theory.  
  The strategy of \cite{CLS23} was further developed in subsequent works. On one hand, Wei-Xu-Zhang \cite{Wei-Xu-Zhang2024}  got the optimal   bound 
  $$\Theta^{1,*}(\infty)\le (\# \mbox{ends}) \cdot|\mathbb S^{2}| 
  $$ for   the volume growth Conjecture \ref{volume growth conjecture} when $n=3$ (independently by Huang-Liu \cite{Huang-Liu2024} with a different method). On the other hand, Wang \cite{Wangyipeng2024} gave an alternative proof for Gromov's volume bound conjecture in dimension three.

  As a natural generalization of scalar curvature, the bi-Ricci curvature was introduced by Shen-Ye in their work \cite{Shen-Ye1996}. Motivated by Schoen–Yau's work \cite{Schoen-Yau83} on diameter bound of stable minimal surfaces in manifolds $M^3$ with positive scalar curvature, they proved similar diameter bounds for stable minimal hypersurfaces in $M^{n}$ with $n\leq 5$ and $\bRic\ge n-2$. Recently,  the research on bi-Ricci curvature plays an important role in the resolution of the stable Bernstein problem \cite{CLMS24, Mazet2024}  and the generalized Geroch conjecture \cite{BHJ24}. 

  Back to the discussion on interplay between volume and curvature, the volume bound conjecture \ref{volume bound conjecture} and the volume growth conjecture \ref{volume growth conjecture} can be asked in a similar way under the condition of positive bi-Ricci curvature lower bound. In their recent work, Antonelli-Xu \cite{Antonelli-Xu2024} showed that a complete non-compact manifold $(M^n,g)$ with dimension $n\in \{3,4,5\}$ satisfying $\Ric\geq 0$ and $\bRic \geq n-2$ must have $\Theta^{1*}(\infty)<\infty$. (The optimal estimate $\Theta^{1*}(\infty)\leq (\#\mbox{ends})\cdot |\mathbb S^{n-1}|$ was also obtained with an extra non-collapsing assumption in \cite{Antonelli-Xu2024} and they mentioned  %we notice
  that the non-collapsing assumption might be removed by using the argument from \cite{Wei-Xu-Zhang2024} or \cite{Huang-Liu2024}). As will be mentioned later, the argument of Antonelli-Xu relies heavily
  on the diameter bound estimate for stable $\mu$-bubbles, which leads to the dimension restriction $n\leq 5$ for the same reason as in Shen-Ye's early result.  It seems that the diameter bound estimate cannot be expected when $M$ has dimension greater than five. Therefore, Antonelli-Xu asked the following
  \begin{ques}\label{Ques: main}
      Let $n\geq 6$ be a natural number. Is it possible to construct a smooth complete non-compact Riemannian manifold such that $\Ric\geq 0$ everywhere, $\bRic\geq n-2$ outside a compact set, and
      \begin{itemize}
          \item either $M$ doesn't have linear volume growth;
          \item or $\lim_{v\to+\infty} I_M(v)>|\mathbb S^{n-1}|$ (when $M$ has only one end)?
      \end{itemize}
  \end{ques}

In this paper, we give a negative answer to Question \ref{Ques: main} when the dimension is less than eight. Namely we have
\begin{theorem}\label{Thm: main 2}
Let $3\leq n\leq 7$. Assume $(M^n,g)$ is a complete and non-compact Riemannian manifold with Ricci curvature $\Ric\ge 0$ and bi-Ricci curvature $\bRic\ge n-2$ outside a compact subset, then  for any $p\in M$, the limit 
\begin{align}\label{existence of density}
\Theta^1(\infty,p)=\lim_{R\to \infty}\frac{\vol(B_R(p))}{R}
\end{align}
exists and is independent of the choice of the base point $p$ (so we can omit the symbol $p$ for short). Moreover, we have the following alternative
\begin{enumerate}[(1)]
\item either $M$ splits as a Riemannian product $N^{n-1}\times \mathbb R$, where $N$ is  a closed Riemannian manifold with $Rc\geq n-2$, and in particular we have  
$$ \Theta^1(\infty)=2\vol(N)\le 2|\mathbb{S}^{n-1}|,$$ 
\item or $M$ has only one end and 
\begin{equation}\label{Eq: volume growth one end}
    \Theta^1(\infty)\le |\mathbb{S}^{n-1}|.
\end{equation}
\end{enumerate}
\end{theorem}
\begin{remark}
    The existence of the limit $\Theta^1(\infty)$ seems to be firstly pointed out  in this work.
\end{remark}

Our Theorem \ref{Thm: main 2} above indeed answers the previous question of Antonelli-Xu since the estimate of isoperimetric profile follows directly from the estimate of volume growth.

\begin{corollary}\label{Cor: iso}
    Let $3\leq n\leq 7$. Assume $(M^n,g)$ is a complete, non-compact and one-ended Riemannian manifold with Ricci curvature $\Ric\ge 0$ and bi-Ricci curvature $\bRic\ge n-2$ outside a compact subset, then the isoperimetric profile $I_M$ satisfies
    $$\lim_{v\to +\infty} I_M(v)\leq |\mathbb S^{n-1}|.$$
\end{corollary}

We point out that previously known arguments in \cite{CLS23,Wei-Xu-Zhang2024, Antonelli-Xu2024} to prove the volume growth conjecture rely heavily on the diameter bound estimate of stable $\mu$-bubbles in manifolds with positive scalar curvature or positive bi-Ricci curvature. When the given manifold is non-splitting, the basic idea is to dice $M$ into tubes for which one can derive nice volume control. In detail, one can fix a large positive constant $L$ and the desired tubes are given by $T_k:=E_{kL}\setminus \bar E_{(k+1)L}$ for $k\in \mathbb N_+$, where $E_{kL}$ is the unique unbounded component of $M\setminus B_{kL}(p)$ for an interior point $p$. Since $M\setminus B_{kL}(p)$ may have bounded components, the tubes can a priori have long necks going deep away from the point $p$, which leads to great trouble in volume control. The solution for this issue is to use the diameter bound estimate of stable $\mu$-bubbles under the condition of positive bi-Ricci curvature. The almost splitting theorem guarantees that the tube $T_k$ is close to a cylinder for large $k$. By constructing a connected separating stable $\mu$-bubble in $T_k$ and applying the diameter bound estimate of stable $\mu$-bubble, one can finally ensure tubes $T_k$ to have bounded size and so the Bishop-Gromov volume comparison theorem gives the bounded volume of $T_k$. Here we emphasize once again that this argument only works when the dimension is no greater than five.

Now we explain our method to obtain the volume growth control in dimensions $6$ and $7$. The crucial point of our argument is to obtain the following volume bound estimate without using the diameter bound estimate of stable $\mu$-bubbles.
\begin{theorem}\label{Thm: main 1}
Let $3\leq n\leq 7$. Assume $(M^n,g)$ is a complete and non-compact Riemannian manifold with Ricci curvature $\Ric\ge 0$ and bi-Ricci curvature $\bRic\ge n-2$ outside a compact subset, then we have
\begin{align}\label{Eq: volume bound estimate}
\sup_{p\in M, R>0}\frac{\vol(B_R(p))}{R}\le 2|\mathbb{S}^{n-1}|,
\end{align}
where $|\mathbb S^{n-1}|$ is denoted to be the volume of the unit $(n-1)$-sphere $\mathbb S^{n-1}(1)\subset \mathbb R^n$.
\end{theorem}

The splitting case is simple and so we only need to focus on the non-splitting case. Our strategy to prove Theorem \ref{Thm: main 1} is to look backward from a far-away point $q$. We use the simple fact that the geodesic ball $B_R(p)$ is contained in some annulus $A_{L,L+2R}(q):=B_{L+2R}(q)\setminus \bar B_L(q)$ with width $2R$. When $q$ is sufficiently near infinity, the constant $L$ will be large enough to find a connected separating stable $\mu$-bubble $\Sigma$ in the annulus $A_{L/2,L}(q)$, for which we can derive the area estimate $\area(\Sigma)\leq (1+o(1))\cdot|\mathbb S^{n-1}|$ as $L\to+\infty$ from Antonelli-Xu's spectrum Bishop-Gromov volume comparison theorem \cite{Antonelli-Xu2024} when the dimension is no greater than seven. Looking backward from the point $q$ we see that the geodesic ball lies behind the $\mu$-bubble $\Sigma$ and then the comparison theorem for hypersurfaces yields that $\vol(B_R(p))\leq 2R\cdot (1+o(1))|\mathbb S^{n-1}|$ as $L\to+\infty$. Here we point out that the coefficient $2R$ comes from the width of the annulus $A_{L,L+2R}(q)$.

%a single connected component of $M\backslash B_L(q)$ for large positive constant $L$. and so it must lie behind some sonnected stable $\mu$-bubble between $q$ and $\partial B_L(q)$. This allows us to apply Antonelli-Xu's spectrum Bishop-Gromov volume comparison theorem \cite{Antonelli-Xu2024} to get the volume bound of the $\mu$-bubble, and we do not need the diameter bound. With the help of $\Ric\ge 0$ assumption, we can verify the spectral Ricci condition for manifolds with $\bRic\ge(n-2)$ up to $n\le 7$. 

Once the volume bound estimate is proved, it follows from the work \cite{Sormani1998} of Sormani that any Busemann function $b$ has a minimum point $p$. Viewing the Busemann function as the distance function from infinity, ideally the geodesic ball $B_R(p)$ will be contained in the annulus $b^{-1}([t,t+R])$ for some constant $t$. Notice that the width of the annulus is improved to $R$ from $2R$. By repeating our previous argument we can obtain the desired volume growth estimate \eqref{Eq: volume growth one end}.

The volume bound estimate \eqref{Eq: volume bound estimate} is optimal for complete non-compact  manifolds with $\Ric\ge0$ and $\bRic\ge n-2$ given the one-ended example from \cite{Wei-Xu-Zhang2024}. After taking closed manifolds into consideration we are also able to verify the volume bound conjecture under the condition of positive bi-Ricci curvature. 

\begin{theorem}\label{Thm: volume conjecture biRic}
    Let $3\le n\leq 7$. There exists a constant $C(n)>0$ such that for any Riemannian manifolds $(M^n,g)$ with Ricci curvature $\Ric\ge 0$ and bi-Ricci curvature  $\bRic\ge n-2$, there holds 
 \begin{align*}
 \sup_{p\in M, R>0}\frac{\vol(B_R(p))}{R}\le C(n).
 \end{align*}
\end{theorem}

\bigskip
\noindent {\it Acknowledgement.} We are grateful to Dr. Kai Xu for reminding us a computational error in the proof of Theorem \ref{Thm: main 1} when the first version of this paper was put online.

 \section{Proof of main theorems}
For any two orthonormal vectors $v,w\in T_pM$,  the bi-Ricci curvature is defined by 
  \begin{align*}
  \bRic(v,w)=\Ric(v,v)+\Ric(w,w)-\Rm(v,w,w,v).
  \end{align*}

\begin{proof}[Proof of Theorem \ref{Thm: main 1}]
We divide the proof into two cases:

\vspace{2mm}{\it Case 1. $M$ satisfies $H_{n-1}(M)\neq 0$.} 

\vspace{2mm}From the work \cite{Shen-Sormani2001} of Shen and Sormani, we know
\begin{itemize}
    \item if $M$ is orientable, then $M$ splits as a Riemannian product $N\times \mathbb R$;
    \item if $M$ is non-orientable, then $M$ is a one-ended flat normal bundle of a closed totally-geodesic orientable submanifold $N$.
\end{itemize}
In the former case, the slice $N$ must be totally geodesic and so it satisfies $\Ric_N\geq n-2$ from our bi-Ricci assumption. In particular, we have $\vol(N)\leq |\mathbb S^{n-1}|$. Since any geodesic $R$-ball is contained in $N\times[l,l+2R]$ for some constant $l\in \mathbb R$, we see that its volume cannot exceed $2R\cdot|\mathbb S^{n-1}|$. In the latter case, the proof of \cite[Theorem 1.2]{Shen-Sormani2001} and \cite[Theorem 11]{Sormani 2001} implies that $M$ has a double cover which splits isometrically. Denote $\pi: \hat{M}\to M$ to be the corresponding covering map and choose $\hat{p}\in \pi^{-1}(p)$. Then we have
$$\vol(B_{R}(p))\leq \vol(B_R(\hat p))\leq 2R\cdot|\mathbb S^{n-1}|.$$

\vspace{2mm}{\it Case 2. $M$ satisfies $H_{n-1}(M)= 0$.}

\vspace{2mm}Take a geodesic ball $B_R(p)\subset M$ and let $\varepsilon>0$ be a constant in $(0,1)$. Take a point $q\in M$ such that $\dist(p,q)>R\varepsilon^{-1}$.  Then we have
\begin{align}\label{Eq: 1}
B_R(p)\subset A_{\dist(p,q)-R,\dist(p,q)+R}(q).
\end{align}
Here and in the sequel, we use $A_{r_1,r_2}(q)$ to denote the open annulus $B_{r_2}(q)\setminus \bar B_{r_1}(q)$ centered at the point $q$.
Let $L=(1-\varepsilon)\cdot \dist(p,q)$. Since we have $R< \varepsilon \cdot \dist(p,q)$ by our choice of $q$, then for any constant $t\in (0,L]$ there holds 
$$B_R(p)\subset M\setminus \bar B_t(q).$$
Denote $E^p_t$ to be the connected component of  $M\setminus \bar B_t(q)$ containing the point $p$. It is clear that we have
$B_R(p)\subset E^p_t$ for all $t\in(0,L]$.

We claim that the boundary $\partial E^p_t$ must be connected for each $t\in(0,L]$. That is, one cannot find two disjoint open subsets $U_1$ and $U_2$ in $M$ such that $\partial E_t^p$ is contained in the union $U_1\cup U_2$ but $\partial E_t^p\cap U_i$ are both non-empty for $i=1,2$. 

Otherwise, the boundary $\partial E_t^p$ will have a decomposition $\partial E_t^p=\Sigma_1\cup \Sigma_2$ with $\Sigma_i=\partial E_t^p\cap U_i\neq \emptyset$ for $i=1,2$. It is clear that both $\Sigma_i=\partial E_t^p\cap \bar U_i$ are compact subsets and in particular there are positive constants $\varepsilon_i$ such that $B_{\varepsilon_i}(\Sigma_i)\subset U_i$. 
Fix a point $q_i$ in $B_{\varepsilon_i}(\Sigma_i)\cap E_t^p$ and let $\gamma_i$ be a minimizing geodesic segment connecting $q$ and $q_i$. When $\varepsilon_i$ is chosen to be small enough such that $\dist(\partial B_{\varepsilon_i}(\Sigma_i),\partial U_i)>\varepsilon_i$, then the minimizing segment $\gamma_i$ intersects $\Sigma_i$ at a unique point $q_i^*$. Take a small portion of $\gamma_i$ centered at the point $q_i^*$ and denote $\gamma_i^*:[-\varepsilon_i^*,\varepsilon_i^*]\to M$ to be its arc-length parameterization such that $\gamma_i^*(0)=q_i^*$, $\gamma_i^*(\varepsilon_i^*)\in E_t^p$ and $\gamma_i^*(-\varepsilon_i^*)\in B_{t}(q)$. Since both $E_t^p$ and $B_t(q)$ are connected, we can connect $\gamma_1^*(\varepsilon_1^*)$ and $\gamma_2^*(\varepsilon_2^*)$ by a path $\zeta_+$ contained in $E_t^p$, and also connect $\gamma_1^*(-\varepsilon_1^*)$ and $\gamma_2^*(-\varepsilon_2^*)$ by a path $\zeta_-$ contained in $B_t(q)$. Denote $\gamma$ to be the closed curve given by $$\gamma=\gamma_1^**\zeta_+*(\gamma_2^*)^{-1}*(\zeta_-)^{-1}.$$
Consider the neighborhood $B_\delta(\Sigma_1)$ with $\delta$  a positive constant to be determined and the signed distance function $\rho_1:B_\delta(\Sigma_1)\to (-\delta,\delta)$ defined by
\[
\rho_1(x)=\left\{\begin{array}{cc}
  \dist(x,\partial E_t^p),   & x\in E_t^p; \\
  -\dist(x,\partial E_t^p),   & x\notin E_t^p.
\end{array}\right.
\]
Since $\gamma$ is a minimizing geodesic segment, by taking $\delta$ small enough we can guarantee
\begin{itemize}
\item  $\rho_1$ is a continuous function on $B_\delta(\Sigma_1)$;
    \item $\im \gamma\cap B_{\delta}(\Sigma_1)=\gamma_1^*((-\delta,\delta))$;
    \item  there is a neighborhood $\mathcal N$ of $\gamma_1^*((-\delta,\delta))$ such that $\rho_1$ is smooth in $\mathcal N$ and equals to $\dist(\cdot,q)-t$.
\end{itemize}
By smoothing we can find a smooth function $\rho_1^*:B_\delta(\Sigma_1)\to (-2\delta,2\delta)$ such that we have $|\rho_1^*-\rho_1|<\delta/2$ in $B_\delta(\Sigma_1)$ and $\rho_1^*=\rho_1$ in $\mathcal N$. Since $\rho_1^*$ is smooth, it follows from the Sard theorem that there is a regular value $\tau\in (-\delta/2,\delta/2)$ such that the preimage $\mathcal S:=(\rho_1^*)^{-1}(\tau)$ is an embedded hypersurface in $B_\delta(\Sigma_1)$. It is clear that we have
$$\mathcal S\cap \im\gamma=\mathcal S\cap \mathcal N\cap \gamma_1^*((-\delta,\delta))=\gamma_1^*(\tau).$$
Namely, $\mathcal S$ intersects transversely with $\gamma$ at a single point $\gamma_1^*(\tau)$. This means that $\mathcal S$ has non-zero intersection number with $\gamma$ and so $[\mathcal S]\neq 0\in H_{n-1}(M)$, which leads to a desired contradiction.

In next step, we want to construct a connected $\mu$-bubble as a shielding hypersurface in our later use of the volume comparison theorem of hypersurfaces. Let $l=(1-5\varepsilon)\cdot\dist(p,q)$ and we work with $V=E_l^p\setminus \bar E_L^p$. Without loss of generality, for fixed $\varepsilon>0$, we may choose $d(p,q)$ large enough such that $\bRic\ge 0$ on $V$.

First we show that $V$ is connected. To see this, we decompose $V$ into the union of its components as
$V=\cup_{i} V_i$, where each $V_i$ is non-empty, open and connected. Notice that we have $\partial E_t^p \subset V$ for all $t\in (l,L)$. For each $t$ the connectedness of $\partial E_t^p$ yields $\partial E_t^p\subset V_{i(t)}$ for some index $i(t)$. Take a point $q_t\in \partial E_t^p$ and let $\gamma_t:[0,t]\to M$ be a minimizing geodesic segment connecting $q$ and $q_t$. Then one can verify $\gamma_t(s) \in \partial E_s^p$ for all $s\in (0,t]$. In particular, this yields $i(t)\equiv i_0$ for all $t\in (l,L)$ and some fixed index $i_0$. Given each point $q_V\in V$ we can connect it to the point $p$ by a path in $E_l^p$ and so in $E_{l+\delta}^p$ with some positive constant $\delta$. Let $\gamma_V:[0,t_V]\to M$ be a minimizing geodesic segment connecting $q$ and $q_V$, then one can verify $\gamma_V(s) \in \partial E_s^p$ for all $s\in (0,l+\delta]$. This yields $V\subset V_{i_0}$ and so $V=V_{i_0}$ is connected.

For convenience, we denote $\rho$ to be the restriction of $\dist(\cdot,q)$ on $V$. Denote $\gamma_p:[0,t_0]\to M$ to be a minimizing geodesic segment connecting $q$ and $p$. By smoothing we can find a smooth function $\rho^*$ on $V$ such that $|\rho^*-\rho|<\epsilon\cdot \dist(p,q)$ and $|\mathrm d\rho^*|<2$ in $V$, and $\rho^*=\rho$ in a small neighborhood of $\gamma_p\left(\frac{l+L}{2}\right)$. As before, for some constant $\tau$ very close to $\frac{l+L}{2}$ the hypersurface $\mathcal S_\tau:=(\rho^*)^{-1}(\tau)$ has non-zero intersection number with the path $\gamma_p|_{[l,L]}$ in $\bar V$. 

Define 
$$V^*=(\rho^*)^{-1}\left(\left[\frac{3l+L}{4}, \frac{l+3L}{4}\right]\right).$$
Let
\begin{equation}\label{Eq: alpha}
    \alpha=\left\{
\begin{array}{cc}
1,&n=3;\\
\frac{n-3}{n-2},&4\leq n\leq 7.
\end{array}\right.
\end{equation}
Denote
\begin{equation}\label{Eq: t_alpha}
    \mathfrak t_\alpha=\frac{1}{n-1}\left(1-\frac{\sqrt{n-2}\alpha}{2}\right)
\end{equation}
and define
$$h=-\mu\tan\left(\frac{\mathfrak t_\alpha \mu}{4}t-\frac{l+L}{2}\right)\mbox{ with }\mu=\frac{8\pi}{\mathfrak t_\alpha (L-l)}.$$
We point out the fact that $n\leq 7$ is used here to guarantee $\mathfrak t_\alpha>0$. It is easy to verify that $h$ satisfies $h'<0$ and
$$\mathfrak t_\alpha h^2+4h'=-\mathfrak t_\alpha \mu^2.$$
Denote $\Omega_0=\{\rho^*<\tau\}\subset V$ and we consider the functional
$$\mathcal A^h(\Omega)=\mathcal H^{n-1}(\partial^*\Omega)-\int_{V}(\chi_\Omega-\chi_{\Omega_0})h\circ \rho^*\,\mathrm d\mathcal H^n,$$
where $\Omega$ is any Caccippoli set in $V$ such that the symmetric difference $\Omega\Delta \Omega_0$ has compact closure in $V^*$ and $\partial^*\Omega$ is its reduced boundary. Since we have $n\leq 7$, it follows from \cite[Proposition 2.1]{Zhu2021} that there is a smooth minimizer $\Omega_{min}$ of the functional $\mathcal A^h$. By definition $\partial\Omega_{min}$ is homologous to $\mathcal S_\tau$ and so it has non-zero intersection number with $\gamma_p|_{[l,L]}$ in $\bar V$. Since the intersection number is additive, non-zero intersection number also holds for some component of $\partial\Omega_{min}$, denoted by $\Sigma^o$. 

Now we verify that  
the hypersurface $\Sigma^o$ satisfies the following properties:
\begin{itemize}
\item (separating property) Any path $\gamma$ in $\bar V$ connecting $\partial E^p_l$ and $\partial E^p_L$ has non-empty intersection with $\Sigma^o$. Otherwise, there is a path $\gamma:[a,b]\to \bar V$ such that $\gamma(a)\in \partial E^p_l$ and $\gamma(b)\in \partial E_L^p$, which does not intersect with $\Sigma^o$. Since $\Sigma^o$ does not touch $\partial E^p_l$ and $\partial E_L^p$, we can find an open neighborhood $U$ of $\partial E_l^p\cup \partial E^p_L$ such that $\Sigma^o\cap U=\emptyset$. The connectedness of $\partial E_l^p$ yields that we can find path $\zeta_l$ in $U$ connecting $\gamma_p(l)$ and $\gamma(a)$. Similarly, we can connect $\gamma_p(L)$ and $\gamma(b)$ by a path $\zeta_L$ in $U$. In particular, $\Sigma^o$ has non-zero intersection number with the closed curve $\gamma_p|_{[l,L]}*\zeta_L*\gamma^{-1}*\zeta_l^{-1}$. This yields $[\Sigma^o]\neq 0\in H_{n-1}$, which is impossible from our assumption.
\item (area bound) The area of $\Sigma^o$ satisfies
$$\area(\Sigma^o)\leq |\mathbb S^{n-1}|\cdot\left(1-\frac{\mathfrak t_\alpha \mu^2}{(n-2)\alpha}\right)^{1-n}.$$
To see this we use the minimizing property of $\Omega_{min}$.  The first variation formula of $\mathcal A^h$ yields 
$$\int_{\Sigma^o}(H-h\circ \rho^*)\psi\,\mathrm d\sigma=0\mbox{ for all }\psi\in C^\infty(\Sigma^o), $$
where $H$ is the mean curvature of $\Sigma^o$ with respect to the outer unit normal $\nu$ of $\Sigma^o$ in $\Omega_{min}$. In particular, we have $H=h\circ \rho^*$ on $\Sigma^o$. The second variation formula of $\mathcal A^h$ implies
$$\int_{\Sigma^o}|\nabla\psi|^2-\left(\Ric(\nu)+|A|^2+2\partial_\nu(h\circ \rho^*)\right)\psi^2\,\mathrm d\sigma\geq 0\mbox{ for all }\psi\in C^\infty(\Sigma^o).$$
Let $\{e_i\}_{i=1}^{n-1}$ be any orthonormal frame on $\Sigma^o$. Recall by definition that $\bRic(\nu,e_1)=\Ric(\nu)+\Ric(e_1)-\Rm(\nu,e_1,e_1,\nu)$. From the Gauss equation we see
$$\Ric(\nu)=\bRic(\nu,e_1)-\Ric_{\Sigma^o}(e_1)+\sum_{i=2}^{n-1}\left(A_{11}A_{ii}-A_{1i}^2\right).$$
Clearly we have
$$\left|\sum_{i=2}^{n-1}A_{11}A_{ii}\right|\leq \sum_{i=2}^{n-1}\left(\frac{1}{2\sqrt{n-2}}A_{11}^2+\frac{\sqrt{n-2}}{2}A_{ii}^2\right)= \frac{\sqrt{n-2}}{2}\sum_{i=1}^{n-1}A_{ii}^2$$
and
$$\sum_{i=2}^{n-1}A_{1i}^2\leq \frac{1}{2}\sum_{i\neq j}A_{ij}^2.$$
Combined with the facts $\Ric(\nu)\geq 0$ and $n\geq 3$ it follows
$$\Ric(\nu)+|A|^2\geq \alpha\left( \bRic(\nu,e_1)-\Ric_{\Sigma^o}(e_1)\right)+\left(1-\frac{\sqrt{n-2}\alpha}{2}\right)|A|^2.$$
Using the fact $|A|^2\geq \frac{H^2}{n-1}$ and $|\nu(h\circ \rho^*)|>2h'$, finally we arrive at
$$\int_{\Sigma^o}|\nabla\psi|^2+\alpha\Ric_{\Sigma^o}(e_1)\psi^2\geq \int_{\Sigma^o}\left(\alpha\bRic(\nu,e_1)+(\mathfrak t_\alpha h^2+4h')\circ\rho^*\right)\psi^2\,\mathrm d\sigma.$$
for all $\psi\in C^\infty(\Sigma^o)$. It follows from \cite[Theorem 1]{Antonelli-Xu2024} that
$$\area(\Sigma^o)\leq |\mathbb S^{n-1}|\cdot\left(1-\frac{\mathfrak t_\alpha \mu^2}{(n-2)\alpha}\right)^{1-n}.$$
\end{itemize}

Then we are ready to derive the volume bound for $B_R(p)$. Denote $\mathcal C_q$ to be the cut locus of $q$. From \eqref{Eq: 1} we know that for any point $x\in B_R(p)\setminus \mathcal C_q$ there is a minimizing geodesic segment $\gamma_x:[0,s]\to M$ with $\dist(p,q)-R<s<\dist(p,q)+R$ connecting points $q$ and $x$. In particular, $\gamma_x|_{[l,L]}$ is a path in $\bar V$ connecting $\partial E_l$ and $\partial E_L$ and so the separating property of $\Sigma^o$ yields $\gamma_x\cap \Sigma^o\neq \emptyset$. In other words, $B_R(p)$ lies behind $\Sigma^o$ with respect to the point $q$. Denote $T_s=\partial B_s(q)\cap (B_R(p)\setminus \mathcal C_q)$. We conclude from the volume comparison theorem of hypersurfaces (see \cite[Lemma 5.3]{Wei-Xu-Zhang2024} for instance) that 
\begin{align*}
\frac{\area(T_s)}{\area(\Sigma^o)}\le \left(\frac{s}{l}\right)^{n-1}\le \left(\frac{\dist(p,q)+R}{(1-5\varepsilon)\cdot\dist(p,q)}\right)^{n-1}.
\end{align*}
As a result, we can compute
\begin{align*}
\vol(B_R(p))&= \int_{\dist(p,q)-R}^{\dist(p,q)+R}\area(T_s)\,\mathrm ds\\
&\le 2R\left(\frac{\dist(p,q)+R}{(1-5\varepsilon)\cdot\dist(p,q)}\right)^{n-1}\area(\Sigma^o).
\end{align*}
Letting $q\to \infty$ and then $\varepsilon\to 0$, we get 
\begin{align*}
\frac{\vol(B_R(p))}{R}\le 2|\mathbb{S}^{n-1}|. 
\end{align*}
Taking supremum for all $p\in M$ and $R>0$, we get 
\begin{align*}
\sup_{p\in M, R>0}\frac{\vol(B_R(p))}{R}\le 2|\mathbb{S}^{n-1}|.
\end{align*}
This completes the proof.
\end{proof}

\begin{proof}[Proof of Theorem \ref{Thm: volume conjecture biRic}]
Denote $\pi: \tilde{M}\to M$ by the universal cover of $M$ and $\tilde{p}=\pi^{-1}(p)$. Then, we know 
$\vol(B_R(p))\le \vol(B_R(\tilde{p}))$.
So, without loss of generality, we may assume $\pi_1(M)=0$. 

If $M$ is non-compact, from Theorem \ref{Thm: main 1} we have
\begin{equation}\label{Eq: bound 1}
    \frac{\vol(B_R(p))}{R}\leq 2|\mathbb S^{n-1}|.
\end{equation}

In the following, we assume $M$ to be closed. Since $\pi_1(M)=0$ implies $M$ is oriented and $H_1(M)=0$, the Poincar\'e duality yields $H_{n-1}(M)=0$.
%In the following, we assume $M$ to be closed. If $H_{n-1}(M)\neq 0$, then the Poincar\'e duality yields that $M$ has infinite fundamental group. In particular, the universal covering of $M$ is non-compact and so we still have \eqref{Eq: bound 1}.Now we focus on the case when $M$ is closed with $H_{n-1}(M)=0$.
As before, let $\alpha$ and $\mathfrak t_\alpha$ be the constants defined by \eqref{Eq: alpha} and \eqref{Eq: t_alpha}.    Take $L_0=L_0(n)>0$ to be a constant such that the constant
    $\mu:=4\pi\mathfrak t_\alpha^{-1}L_0^{-1}$ satisfies
    $$\frac{\mathfrak t_\alpha \mu^2}{(n-1)\alpha}=\frac{1}{2}.$$
    The following discussion will be divided into two cases:

    {\it Case 1. $\diam M\leq 6L_0$.} It follows from the Bishop-Gromov volume comparison theorem that
    \begin{equation}\label{Eq: bound 2}
        \frac{\vol(B_R(p))}{R}\leq \omega_n\left(\min\{R,\diam M\}\right)^{n-1}\leq \omega_n(6L_0)^{n-1},
    \end{equation}
    where $\omega_n$ is denoted to be the volume of the unit Euclidean $n$-ball.

    {\it Case 2. $\diam M>6L_0$.} Take point $q$ such that 
    $$\dist(p,q)=\max_{x\in M}\dist(p,x).$$
    From $\diam M>6L_0$ we see $\dist(p,q)>3L_0$. Denote $r_0=\min\{\dist(p,q),R\}$. From the Bishop-Gromov comparison theorem we know
    $$\vol(B_R(p))=\vol(B_{r_0}(p))\leq 3^n\vol(B_{r_0/3}(p)).$$
    It suffices to estimate $\vol(B_{r_0/3}(p))$. To do this we construct a connected $\mu$-bubble in the annulus region $A_{d(p,q)-2r_0/3,d(p,q)-r_0/3}(q)$ whose area cannot exceed
    $2^{n-1}|\mathbb S^{n-2}|$. Denote $T_s=\partial B_s(q)\cap B_{r_0/3}(p)$. Using the volume comparison theorem of hypersurfaces we obtain
    \[
    \begin{split}
    \vol(B_{r_0/3}(p))&=\int_{d(p,q)-\frac{r_0}{3}}^{d(p,q)+\frac{r_0}{3}}\area(T_s)\,\mathrm ds\\
    &\leq \int_{d(p,q)-\frac{r_0}{3}}^{d(p,q)+\frac{r_0}{3}} 2^{n-1}|\mathbb S^{n-1}|\cdot\left(\frac{3s}{3d(p,q)-2r_0}\right)^{n-1}\,\mathrm ds\\
    &\leq \frac{2}{3}\cdot 8^{n-1}|\mathbb S^{n-1}|r_0.
    \end{split}\]
    From this we see
    \begin{equation}\label{Eq: bound 3}
        \frac{\vol(B_R(p))}{R}\leq 2\cdot (24)^{n-1}|\mathbb S^{n-1}|.
    \end{equation}

    The proof is finally completed by combining \eqref{Eq: bound 1}, \eqref{Eq: bound 2} and \eqref{Eq: bound 3}.
\end{proof}

\begin{proof}[Proof of Theorem \ref{Thm: main 2}]
If $M$ has at least two ends, then by the splitting theorem $M$ splits as the Riemannian product $N^{n-1}\times \mathbb R$, where $N$ is a closed Riemannian manifold with $\Ric\geq n-2$. In particular, we have $\area(N)\leq |\mathbb S^{n-1}|$.  Now it is easy to verify that all the desired properties hold. 

In the following, we may assume $M$ to have only one end and we divide the discussion into two cases:

{\it Case 1. $H_{n-1}(M)\neq 0$.} As before, $M$ has a double cover which splits isometrically as $N\times \mathbb R$. Denote $\pi: \hat{M}\to M$ to be the corresponding covering map and the preimage $\pi^{-1}(p)$ consists of two points, denoted by $\hat p_1$ and $\hat p_2$. Clearly we have 
$$\vol(B_R(p))=\frac{1}{2}\vol(\pi^{-1}(B_R(p)))=\frac{1}{2}\vol(B_R(\hat p_1)\cup B_R(\hat p_2)).$$
Using the relation $B_R(\hat p_1)\subset B_R(\hat p_1)\cup B_R(\hat p_2)\subset B_{R+\dist(\hat p_1,\hat p_2)}(\hat p_1)$ we obtain
$$\frac{1}{2}\vol(B_R(\hat p_1))\leq \vol(B_R(p))\leq \frac{1}{2}B_{R+\dist(\hat p_1,\hat p_2)}(\hat p_1).$$
Multiplying the factor $R^{-1}$ to this inequality and letting $R\to+\infty$ we see
$$\lim_{R\to+\infty} \frac{\vol(B_R(p))}{R}=\area(N)\leq |\mathbb S^{n-1}|.$$

{\it Case 2. $H_{n-1}(M)=0$.}
In this case, by Theorem \ref{Thm: main 1} we already know that $M$ has linear volume growth, i.e.  
$$\limsup_{R\to \infty}\frac{\vol(B_R(p))}{R}\le 2|\mathbb{S}^{n-1}|.$$
From Sormani's theorem on the properness of the Busemann function \cite[Theorem 19 and Corollary 23]{Sormani1998} applied to Ricci-nonnegative manifolds with linear volume growth, we know that for any geodesic ray $\gamma:[0,+\infty)\to M$ on $M$ the corresponding Busemann function given by 
$$b(x)=\lim_{t\to +\infty}(t-\dist(x,\gamma(t)))$$ is proper and it satisfies
\begin{align*}
b_{min}:=\inf_{x\in M} b(x)>-\infty.
\end{align*}
In particular, there exists a point $p_0\in M$ such that $b(p_0)=b_{min}$. 

Let $q_t=\gamma(t)$ with $t$ large. Recall that in the proof of Theorem \ref{Thm: main 1} we use the relation $B_R(p_0)\subset A_{\dist(p_0,q_t)-R,\dist(p_0,q_t)+R}(q)$. Here we want to make an improvement on this using the fact that $p_0$ is a minimum point of the Busemann function $b$. Define 
$$b_{t}(x)=t-d(x,\gamma(t)).$$ 
It is well-known that $b_t(x)$ converges to $b(x)$ uniformly in any compact subset. Take $\bar B_R(p_0)$ as the compact subset. As a consequence, there is a positive constant $t_0$ such that $|b_t-b|_{L^\infty(B_R(p_0))}<1$ and so when $t\geq t_0$ we obtain
$$B_R(p_0)\subset b^{-1}([b_{min},b_{min}+R])\subset b_t^{-1}([b_{min}-1,b_{min}+R+1]).$$
Equivalently we have
$$B_R(p_0)\subset A_{t-b_{min}-R-1,t-b_{min}+1}(q_t).$$
Note that the width of the annulus equals to $R+2$, which is better than the previous $2R$ in view of the change of coefficient from two to one.

Now we can repeat the argument in the proof of Theorem \ref{Thm: main 1} for our desired results.

{\it 1. Existence of the limit.}
Denote
\begin{align*}
\Theta^1_{*}(\infty,p_0)=\liminf_{R\to \infty}\frac{\vol(B_R(p_0))}{R}.
\end{align*}
By definition for any $\varepsilon>0$ there exists a sequence $\{R_i\}$ with $R_i\to\infty$ as $i\to\infty$ such that 
\begin{align*}
\vol(B_{R_i}(p_0))\le (\Theta^1_{*}(\infty,p_0)+\varepsilon) R_i.
\end{align*}
From the co-area formula we can construct a closed hypersurface $\Sigma_i$  satisfying
\begin{itemize}
    \item any path in $\bar A_{\epsilon R_i,R_i}(p_0)$ connecting $\partial B_{\epsilon R_i}(p_0)$ and $\partial B_{R_i}(p_0)$ must have non-empty intersection with $\Sigma_i$;
    \item $\area(\Sigma_i)\leq (1-2\varepsilon)^{-1}(\Theta^1_{*}(\infty,p_0)+\varepsilon)$.
\end{itemize}
A similar argument as in the proof of Theorem \ref{Thm: main 1} yields
$$\Theta^{1,*}(\infty,p_0):=\limsup_{R\to \infty}\frac{\vol(B_R(p_0))}{R}\leq (1-2\varepsilon)^{-1}(\Theta^1_{*}(\infty,p_0)+\varepsilon).$$
By letting $\epsilon\to 0$ we see that the limit $\Theta^1(\infty,p_0)$ exists. It is a simple fact that if $\Theta^1(\infty,\cdot)$ exists at one point then it exists at all points, whose value is independent of the chosen point.

{\it 2. Optimal volume growth.} Since $H_{n-1}(M)=0$,  by the same argument as in the proof of Theorem \ref{Thm: main 1}, we know 
$$\vol(B_R(p))\leq (R+2)\left(\frac{\dist(p,q_t)+R}{(1-5\varepsilon)\cdot\dist(p,q_t)}\right)^{n-1}\area(\Sigma^o).$$
Letting $t\to +\infty$, $\varepsilon\to 0$ and $R\to+\infty$ we obtain $\Theta^1(\infty)\leq |\mathbb S^{n-1}|$.
\end{proof}
\begin{remark} By the proof above, we see that the density limit \eqref{existence of density} always exists for complete non-compact manifold with nonnegative Ricci curvature and minimal volume growth. Note that the bi-Ricci assumption is only necessary for the sharp bound.
\end{remark}

Finally, let us give the proof of Corollary \ref{Cor: iso}.
\begin{proof}[Proof of Corollary \ref{Cor: iso}]
    The corollary follows from the simple fact: if $(M,g)$ is a complete and non-compact Riemannian manifold with nonnegative Ricci curvature, where the limit $\Theta^1(\infty)$ exists, then we have 
    $$\lim_{v\to+\infty}I_M(v)=\Theta^1(\infty).$$
    To see this first notice that the isoperimetric profile $I_M(v)$ is non-decreasing due to the nonnegative Ricci curvature (see \cite[Theorem 3.8 (1)]{Antonelli-EPS2022}). Fix any constant $\epsilon$ in $(0,1/2)$ and a point $p\in M$. By definition we can find a sequence $\{r_i\}$ diverging to $+\infty$ such that $\vol(B_{r_i}(p))\leq \Theta^1(\infty)+\epsilon$. From co-area formula we can find a hypersurface $\mathcal S_i$ such that $\mathcal S_i$ encloses $B_{\epsilon r_i}(p)$ and satisfies $\area(\mathcal S_i)\leq (1-2\epsilon)^{-1}(\Theta^1(\infty)+\epsilon)$. Denote $v_i$ to be the volume enclosed by $\mathcal S_i$. Then we obtain
    $$I_M(v_i)\leq \frac{\Theta^1(\infty)+\epsilon}{1-2\epsilon}.$$
    Since any complete and non-compact Riemannian manifold with nonnegative Ricci curvature has infinite volume, we see $v_i\to +\infty$ and so $\lim_{v\to+\infty}I_M(v)\leq (1-2\epsilon)^{-1}(\Theta^1(\infty)+\epsilon)$. The proof of the fact is now completed by letting $\epsilon\to 0$. When $M$ has only one end, we know $\lim_{v\to+\infty}I_M(v)\leq \Theta^1(\infty)\leq |\mathbb S^{n-1}|$.
\end{proof}

\bibliography{bib}
\bibliographystyle{amsplain}	
\end{document}